\documentclass[english,12pt]{article}

\usepackage[T1]{fontenc}
\usepackage{ifthen}
\usepackage{amsmath}
\usepackage{amsfonts}
\usepackage{babel}
\usepackage[sort&compress,square,numbers]{natbib}
\usepackage{enumitem}
\usepackage[draft=false,final,plainpages=false,pdftex]{hyperref}
\usepackage[amsmath,thmmarks,hyperref]{ntheorem}

\usepackage{txfonts}
\let\mathbb=\varmathbb

\provideboolean{usemicrotype}
\setboolean{usemicrotype}{true}

\ifthenelse{\boolean{usemicrotype}}{
	\usepackage[spacing=true,tracking=true,kerning=true,babel]{microtype}
	\microtypecontext{spacing=nonfrench}
}{}

\theoremnumbering{arabic}
\theoremstyle{plain}
\theoremseparator{.}
\RequirePackage{latexsym}
\theoremsymbol{}
\theorembodyfont{\itshape}
\theoremheaderfont{\normalfont\bfseries}
\newtheorem{theorem}{Theorem}[section]
\newtheorem{proposition}[theorem]{Proposition}
\newtheorem{lemma}[theorem]{Lemma}
\newtheorem{corollary}[theorem]{Corollary}

\theoremheaderfont{\upshape}
\newtheorem{claim}{Claim}[theorem]

\theorembodyfont{\upshape}
\theoremsymbol{}
\theoremheaderfont{\normalfont\bfseries}
\newtheorem{definition}[theorem]{Definition}

\theoremstyle{nonumberplain}
\theoremheaderfont{\itshape}
\theorembodyfont{\normalfont}
\theoremsymbol{\ensuremath{\square}}
\RequirePackage{amssymb}
\newtheorem{proof}{Proof}

\theoremsymbol{\ensuremath{\dashv}}
\newtheorem{claimproof}{Proof}

\qedsymbol{\ensuremath{\square}}
\theoremclass{LaTeX}

\newcommand{\Ord}{\ensuremath{\text{{\rm Ord}}}}
\newcommand{\ZFC}{\ensuremath{\text{{\sf ZFC}}}}
\def\forces{\mathrel\|\joinrel\relbar}
\DeclareMathOperator{\dom}{dom}
\DeclareMathOperator{\rng}{rng}

\DeclareMathOperator{\cf}{cf}
\DeclareMathOperator{\cof}{cof}

\newcommand{\MA}{\ensuremath{\text{{\sf MA}}}}
\newcommand{\PFA}{\ensuremath{\text{{\sf PFA}}}}
\newcommand{\MRP}{\ensuremath{\text{{\sf MRP}}}}
\newcommand{\MM}{\ensuremath{\text{{\sf MM}}}}

\newcommand{\Cl}{\ensuremath{\text{{\rm Cl}}}}
\newcommand{\TP}{\ensuremath{\text{{\sf TP}}}}
\newcommand{\ITP}{\ensuremath{\text{{\sf ITP}}}}
\newcommand{\ISP}{\ensuremath{\text{{\sf ISP}}}}
\newcommand{\SCH}{\ensuremath{\text{{\sf SCH}}}}
\newcommand{\SP}{\ensuremath{\text{{\sf SP}}}}
\newcommand{\non}{\ensuremath{\lnot}}
\newcommand{\IT}{\ensuremath{\text{{\rm IT}}}}
\newcommand{\IS}{\ensuremath{\text{{\rm IS}}}}
\newcommand{\medcup}{\ensuremath{{\textstyle\bigcup}}}

\DeclareMathOperator{\height}{ht}
\DeclareMathOperator{\Con}{Con}

\DeclareMathOperator{\Coll}{Coll}

\begin{document}

\title{On the consistency strength of the proper forcing axiom}
\author{Matteo Viale, Christoph Wei\ss\footnote{Parts of the results of this paper are from the second author's doctoral dissertation~\cite{diss} written under the supervision of Dieter Donder, to whom the second author wishes to express his gratitude.}}
\date{}

\maketitle

\begin{abstract}
Recently the second author introduced combinatorial principles that characterize supercompactness for inaccessible cardinals but can also hold true for small cardinals.
We prove that the proper forcing axiom \PFA\ implies these principles hold for $\omega_2$.
Using this, we argue to show that any of the known methods for forcing models of \PFA\ from a large cardinal assumption requires a strongly compact cardinal.
If one forces \PFA\ using a proper forcing, then we get the optimal result that a supercompact cardinal is necessary.
\end{abstract}



\section{Introduction}

Since their introduction in the seventies supercompact cardinals played a central role in set theory. They have been a fundamental assumption to obtain many of the most interesting breakthroughs:
Solovay's original proof that the singular cardinal hypothesis \SCH\ holds eventually above a large cardinal, Silver's first proof of $\Con(\neg\SCH)$, Baumgartner's proof of the consistency of the proper forcing axiom \PFA~\cite{devlin} and Foreman, Magidor, and Shelah's proof of the consistency of Martin's maximum \MM~\cite{foreman_magidor_shelah} all relied on the assumption of the existence of a supercompact cardinal.

While some of these result have been shown to have considerably weaker consistency strength, the exact large cardinal strength of the forcing axioms \PFA\ and \MM\ is one of the major open problems in set theory.
It is what we want to address in this paper.

Forcing axioms play an important role in contemporary set theory.
Historically they evolved from Martin's axiom, which was commonly used as the axiomatic counterpart to ``$V = L$.''
The most prominent forcing axioms today are \PFA\ as well as the stronger \MM.
Not only do they serve as a natural extension of \ZFC, they also answer a plethora of questions undecidable in \ZFC\ alone,
from elementary questions like the size of the continuum to combinatorially complicated ones like the basis problem for uncountable linear orders~\cite{moore.basis}.
Even problems originating from other fields of mathematics and apparently unrelated to set theory have been settled appealing to \PFA.
For example, Farah~\cite{farah} recently proved the nonexistence of outer automorphisms of the Calkin algebra assuming \PFA.

The consistency proofs of \PFA\ and \MM\ both start in a set theoretic universe in which there is a supercompact cardinal $\kappa$.
They then collapse $\kappa$ to $\omega_2$ in such a way that in the resulting model \PFA\ or \MM\ holds,
thus showing the consistency strength of these axioms is at most that of the existence of a supercompact cardinal.

An early result on \PFA\ by Baumgartner~\cite{baumgartner.PFA} was that \PFA\ implies the tree property on $\omega_2$, that is, \PFA\ implies there are no $\omega_2$-Aronszajn trees.
As a cardinal $\kappa$ is weakly compact if and only if it is inaccessible and the tree property holds on $\kappa$, this can be seen as \PFA\ showing the ``weak compactness'' of $\omega_2$, apart from its missing inaccessibility.
This is an affirmation of the idea that collapsing a large cardinal to $\omega_2$ is necessary to produce a model of \PFA,
and it actually implies the consistency strength of \PFA\ is at least the existence of a weakly compact cardinal,
for if the tree property holds on $\omega_2$, then $\omega_2$ is weakly compact in $L$ by~\cite{mitchell}.

This was the first insight that showed \PFA\ posses large cardinal strength, and many heuristic results indicate that supercompactness actually is the correct consistency strength of \PFA\ and thus in particular also of \MM.
Still giving lower bounds for the consistency strength of \PFA\ or \MM\ is one major open problem today.
While inner model theoretic methods were refined and enhanced tremendously over the last three decades, the best lower bounds they can establish today are still far below supercompactness~\cite{jensen.schimmerling.schindler.steel}.

In~\cite{weiss} the second author introduced combinatorial principles which do for strong compactness and supercompactness what the tree property does for weak compactness:
A cardinal $\kappa$ is strongly compact (supercompact) if and only if $\kappa$ is inaccessible and $\TP(\kappa)$ or, equivalently, $\SP(\kappa)$ ($\ITP(\kappa)$ or, equivalently, $\ISP(\kappa)$) holds.
We will show \PFA\ implies $\ISP(\omega_2)$, the strongest of the four principles. 
This, in the line of thought from above, says \PFA\ shows $\omega_2$ is, modulo inaccessibility, ``supercompact.''

Apart from the strong heuristic evidence this gives, by using arguments for pulling back these principles from generic extensions these characterizations actually allow us to show the following theorems:
If one forces a model of \PFA\ using a forcing that collapses a large cardinal $\kappa$ to $\omega_2$ and satisfies the $\kappa$-covering and $\kappa$-approximation properties,\footnote{See Definition~\ref{def.covering_approximation}.} then $\kappa$ has to be strongly compact;
if the forcing is also proper, then $\kappa$ is supercompact.
We will show that all known forcings for producing models of \PFA\ by collapsing an inaccessible cardinal $\kappa$ to $\omega_2$ satisfy these properties.

Results of this kind have first been obtained by Neeman~\cite{Nee08}.
He showed that if one starts with a ground model that satisfies certain fine structural properties and forces \PFA\ by means of a proper forcing, then $\omega_2$ of the generic extension has to be a cardinal $\kappa$ which is close to being $\kappa^+$-supercompact in the ground model.
(More precisely, in the ground model $[\kappa,\kappa^+]$ is a $\Sigma^2_1$-indescribable gap.)
Our results, which approach the issue from a different perspective, are substantially stronger in that they reach full supercompactness.

\subsection*{Notation}

The notation used is mostly standard.
For a regular cardinal $\delta$, $\cof \delta$ denotes the class of all ordinals of cofinality $\delta$.

The phrases \emph{for large enough $\theta$} and \emph{for sufficiently large $\theta$}
will be used for saying that there exists a $\theta'$ such that the sentence's proposition holds for all $\theta \geq \theta'$.

For an ordinal $\kappa$ and a set $X$ we let $P_\kappa X \coloneqq \{ x \subset X\ |\ |x| < \kappa \}$ and, if $\kappa \subset X$,
\begin{equation*}
	P_\kappa' X \coloneqq \{ x \in P_\kappa X\ |\ \kappa \cap x \in \Ord,\ \langle x, \in \rangle \prec \langle X, \in \rangle  \}.
\end{equation*}
For $x \in P_\kappa X$ we set $\kappa_x \coloneqq \kappa \cap x$.
For $f: P_\omega X \to P_\kappa X$ let $\Cl_f \coloneqq \{ x \in P_\kappa X\ |\ \forall z \in P_\omega x\ f(z) \subset x \}$.
$\Cl_f$ is club, and it is well known that for any club $C \subset P_\kappa X$ there is an $f: P_\omega X \to P_\kappa X$ such that $\Cl_f \subset C$.

For sections~\ref{sect.principles} and~\ref{sect.guessing}, $\kappa$ and $\lambda$ are assumed to be cardinals, $\kappa \leq \lambda$, and $\kappa$ is regular and uncountable.

\subsection*{Acknowledgments}

The authors wish to express their gratitude to
David Asper\'{o},
Sean Cox,
Dieter Donder,
Hiroshi Sakai,
Ralf Schindler,
and Boban Veli\v{c}kovi\'c
for valuable comments and feedback on this research.
They are indebted to Menachem Magidor for supplying them with the idea of the proof of Theorem~\ref{theorem.pull_back_ISP}, that is, Claim~\ref{claim.magidor}.
They furthermore want to thank Mauro Di Nasso for an invitation to discuss this material at a one week workshop in Pisa.

\section{The principles {\sffamily TP}, {\sffamily SP}, {\sffamily ITP}, and {\sffamily ISP}}\label{sect.principles}

We recall the necessary definitions from~\cite{weiss}.
Let us call a sequence $\langle d_a\ |\ a \in P_\kappa \lambda \rangle$ a \emph{$P_\kappa \lambda$-list} if $d_a \subset a$ for all $a \in P_\kappa \lambda$.

\begin{definition}\label{def.P_kappa_lambda.thin}
	Let $D = \langle d_a\ |\ a \in P_\kappa \lambda \rangle$ be a $P_\kappa \lambda$-list.
	\begin{itemize}
		\item $D$ is called \emph{thin} if there is a club $C \subset P_\kappa \lambda$ such that $| \{ d_a \cap c\ |\ c \subset a \in P_\kappa \lambda \} | < \kappa$ for every $c \in C$.
		\item $D$ is called \emph{slender} if for every sufficiently large $\theta$ there is a club $C \subset P_\kappa H_\theta$ such that $d_{M \cap \lambda} \cap b \in M$ for all $M \in C$ and all $b \in M \cap P_{\omega_1} \lambda$.
	\end{itemize}
\end{definition}
Note that if $D$ is a thin list, then $D$ is slender.

\begin{definition}\label{def.ineffable_branch}
	Let $D = \langle d_a\ |\ a \in P_\kappa \lambda \rangle$ be a $P_\kappa \lambda$-list and $d \subset \lambda$.
	\begin{itemize}
		\item $d$ is called a \emph{cofinal branch of $D$} if for all $a \in P_\kappa \lambda$ there is $z_a \in P_\kappa \lambda$ such that $a \subset z_a$ and $d \cap a = d_{z_a} \cap a$.
		\item $d$ is called an \emph{ineffable branch of $D$} if there is a stationary set $S \subset P_\kappa \lambda$ such that $d \cap a = d_a$ for all $a \in S$.
	\end{itemize}
\end{definition}

\begin{definition}
	\begin{itemize}
		\item $\TP(\kappa, \lambda)$ holds if every thin $P_\kappa \lambda$-list has a cofinal branch.
		\item $\SP(\kappa, \lambda)$ holds if every slender $P_\kappa \lambda$-list has a cofinal branch.
		\item $\ITP(\kappa, \lambda)$ holds if every thin $P_\kappa \lambda$-list has an ineffable branch.
		\item $\ISP(\kappa, \lambda)$ holds if every slender $P_\kappa \lambda$-list has an ineffable branch.
	\end{itemize}
	We let $\TP(\kappa)$ abbreviate the statement that $\TP(\kappa, \lambda)$ holds for all $\lambda \geq \kappa$, and similarly for the other principles.
\end{definition}

These definitions admit different ways of defining strong compactness and supercompactness.
\begin{theorem}\label{theorem.TP<->stronglycompact}
	Suppose $\kappa$ is inaccessible.
	Then $\kappa$ is strongly compact if and only if\/ $\TP(\kappa)$ holds.
\end{theorem} 
\begin{theorem}\label{theorem.ITP<->supercompact}
	Suppose $\kappa$ is inaccessible.
	Then $\kappa$ is supercompact if and only if\/ $\ITP(\kappa)$ holds.
\end{theorem}
Unlike other characterizations however, by~\cite{weiss} the principles \ITP\ and \ISP\ also make sense for small cardinals.

There exist ideals and filters naturally associated to the principles \ITP\ and \ISP.
\begin{definition}
	Let $A \subset P_\kappa \lambda$ and let $D = \langle d_a\ |\ a \in P_\kappa \lambda \rangle$ be a $P_\kappa \lambda$-list.
	$D$ is called \emph{$A$-effable} if for every $S \subset A$ that is stationary in $P_\kappa \lambda$ there are $a, b \in S$ such that $a \subset b$ and	$d_a \neq d_b \cap a$.
	$D$ is called \emph{effable} if it is $P_\kappa \lambda$-effable.
\end{definition}
\begin{definition}
	We let
	\begin{align*}
		I_\IT[\kappa, \lambda] & \coloneqq \{ A \subset P_\kappa \lambda\ |\ \text{there exists a thin $A$-effable $P_\kappa \lambda$-list} \},\\
		I_\IS[\kappa, \lambda] & \coloneqq \{ A \subset P_\kappa \lambda\ |\ \text{there exists a slender $A$-effable $P_\kappa \lambda$-list} \}.
	\end{align*}
	By $F_\IT[\kappa, \lambda]$ and $F_\IS[\kappa, \lambda]$ we denote the filters associated to $I_\IT[\kappa, \lambda]$
	and $I_\IS[\kappa, \lambda]$ respectively.
\end{definition}
The ideals $I_\IT[\kappa, \lambda]$ and $I_\IS[\kappa, \lambda]$ are normal ideals on $P_\kappa \lambda$ by~\cite{weiss}.

\section{Guessing models}\label{sect.guessing}

We now introduce the concept of a \emph{guessing model} which gives an alternative presentation of the principle \ISP. 

\begin{definition}
	Let $M \prec H_\theta$ for some large enough $\theta$.
	\begin{itemize}
		\item A set $d$ is called $M$-\emph{approximated} if $d \cap b \in M$ for all $b \in M \cap P_{\omega_1} M$.
		\item A set $d$ is called $M$-\emph{guessed} if there is an $e \in M$ such that $d \cap M = e \cap M$.
	\end{itemize}
	$M$ is called \emph{$z$-guessing} if every $M$-approximated $d \subset z$ is $M$-guessed.
	$M$ is called \emph{guessing} if for all $z \in M$, $M$ is $z$-guessing.
\end{definition}
Note that since for every $z \in M$ there is a bijection $f: z \to \rho$ in $M$ for some ordinal $\rho$, it holds that $M$ is guessing if and only if $M$ is $\rho$-guessing for all $\rho \in M$.
Also note that since $M$ cannot be $\sup (M \cap \Ord)$-guessing, any ordinal $\rho$ such that $M$ is $\rho$-guessing has to be bounded by $\sup(M\cap \Ord)$.

Define
\begin{align*}
	\mathcal{G}^z_\kappa X & \coloneqq \{ M \in P'_\kappa X\ |\ \text{$M$ is $z$-guessing} \},\\
	\mathcal{G}_\kappa X & \coloneqq \{ M \in P'_\kappa X\ |\ \text{$M$ is guessing} \}.
\end{align*}

\begin{proposition}\label{prop.guessingmodels}
	If\/ $\ISP(\kappa, |H_\theta|)$ holds, then $\mathcal{G}_\kappa H_\theta$ is stationary.
\end{proposition}
\begin{proof}
	By working with a bijection $f: |H_\theta| \to H_\theta$, it is obvious that we can apply $\ISP(\kappa, |H_\theta|)$ to the set $P_\kappa H_\theta$ directly.

	Suppose to the contrary that there is a club $C \subset P'_\kappa H_\theta$ such that every $M \in C$ is not guessing, that is, there is $z_M \in M$ and $d_M \subset z_M$ that is $M$-approximated but not $M$-guessed.
	Then also $d_M \cap M$ is $M$-approximated but not $M$-guessed, so we may assume $d_M \subset M$.
	Consider the list $D \coloneqq \langle d_M\ |\ M \in C \rangle$.
	
	Then $D$ is slender, for let $\theta'$ be large enough and let $C' \coloneqq \{ M' \in P_\kappa H_{\theta'}\ |\ M' \cap H_\theta \in C \}$.
	$C'$ is club in $P_\kappa H_\theta$, and if $M' \in C$ and $b \in P_{\omega_1} H_\theta \cap M'$, then $b \in M' \cap H_\theta$, so $d_{M' \cap H_\theta} \cap b \in M' \cap H_\theta \subset M'$.

	By $\ISP(\kappa, |H_\theta|)$, there is an ineffable branch $d$ for the list $D$.
	Let $S \coloneqq \{ M \in C\ |\ d_M = d \cap M \}$.
	$S$ is stationary, and we may assume $z_M = z$ for some fixed $z$ and all $M \in S$.
	This means $d \subset z$.
	As $Pz \subset H_\theta$, there is an $M \in S$ such that $d \in M$.
	But then $d_M$ is $M$-guessed, a contradiction.
\end{proof}

\begin{proposition}\label{prop.guessing->ISP}
	Let $\theta$ be sufficiently large and $M \in P'_\kappa H_\theta$ be a $\lambda$-guessing model such that $\lambda^+ \in M$.
	Then $\ISP(\kappa, \lambda)$ holds. 
\end{proposition}
\begin{proof}
	Since $M\prec H_\theta$ it is enough to show that $M \models \ISP(\kappa, \lambda)$.
	So pick a slender list $D = \langle d_a\ |\ a \in P_\kappa \lambda \rangle \in M$.
	Notice that the slenderness of $D$ is witnessed by a club $C' \subset P_\kappa H_{\lambda^+}$ which is in $M$.
	Then $M \cap H_{\lambda^+} \in C'$, so $d_{M \cap \lambda} \cap b \in M$ for all $b \in M \cap P_{\omega_1} \lambda$.
	This means $d_{M \cap \lambda}$ is an $M$-approximated subset of $M$.
	So since $M$ is a $\lambda$-guessing model, there is an $e \in M$ such that $e \cap M = d_{M \cap \lambda}$.

	Let $S \coloneqq \{ a \in P_\kappa \lambda\ |\ d_a = e \cap a \}$.
	Then $S \in M$.
	To see $S$ is stationary, let $C \in M$ be a club in $P_\kappa \lambda$.
	Then $M \cap \lambda \in C \cap S$, so $H_\theta \models C \cap S \neq \emptyset$, so it also holds in $M$.
\end{proof}

Notice that we cannot literally say that $F_\IS[\kappa, H_\theta]$ is the club filter restricted to $\mathcal{G}_\kappa H_\theta$:
There might be a slender list $\langle d_M\ |\ M\in S\rangle$ indexed by some stationary set $S \subset \mathcal{G}_\kappa H_\theta$ that does not have an ineffable branch.
For such a list we necessarily have that $d_M \not\subset z$ for all $z \in M$ and all $M \in S$.
Still the following holds.
\begin{proposition}
	$I_\IS[\kappa, X]$ is contained in the projection of the nonstationary ideal restricted to $\mathcal{G}_\kappa^X H_\theta$ onto $X$ for any regular $\theta$ such that $X \in H_\theta$.
\end{proposition}
\begin{proof}
	Assume to the contrary that there is an $S \in I_\IS[\kappa, X]$ such that $S^* \coloneqq \{ M \in \mathcal{G}_\kappa^X H_\theta\ |\  M\cap X\in S\}$ is stationary.
	Pick a slender list $D = \langle d_a\ |\ a\in S\rangle$ witnessing that $S \in I_\IS[\kappa, X]$.
	Let $C$ be a club subset of $P_\kappa H_\theta$ witnessing that $D$ is slender.
	Pick $M \in S^*\cap C$ such that $D\in M$.
	Then $d_{M\cap X}$ is an $M$-approximated subset of $X$ as $M\in C$.
	Thus $d_{M\cap X} = e \cap M$ for some $e \in M$ since $M$ is $X$-guessing.
	As in the proof of Proposition~\ref{prop.guessing->ISP} it follows that $e$ is an ineffable branch for $D$, contradicting the fact that $D$ witnesses $S \in I_\IS[\kappa, X]$.
\end{proof}

\section{Implications under {\sffamily PFA}}\label{sect.PFA}

In this section, we are going to show \PFA\ implies $\ISP(\omega_2)$.

The following lemma is due to Woodin~\cite[Proof of Theorem~2.53]{woodin}.
Recall that $G \subset \mathbb{P}$ is said to be \emph{$M$-generic}
if $G$ is a filter on $\mathbb{P}$ and $G \cap D \cap M \neq \emptyset$ for all $D \in M$ that are dense in $\mathbb{P}$.
\begin{lemma}\label{lemma.PFA_stationarily_often}
	Let $\mathbb{P}$ be a proper forcing, and let $\theta$ be sufficiently large.
	Then \PFA\ implies
	\begin{equation*}
		\{ M \in P_{\omega_2} H_\theta\ |\ \exists G \subset \mathbb{P}\ \text{$G$ is $M$-generic} \}
	\end{equation*}
	is stationary in $P_{\omega_2} H_\theta$.
\end{lemma}

\begin{definition}
	Let $T$ be a tree and $B$ be a set of cofinal branches of $T$.
	A function $g: B \to T$ is called \emph{Baumgartner function} if $g$ is injective and
	for all $b, b' \in B$ it holds that
	\begin{enumerate}
		\item $g(b) \in b$,
		\item $g(b) < g(b') \rightarrow g(b') \notin b$.
	\end{enumerate}
\end{definition}

The following lemma is due to Baumgartner, see~\cite{baumgartner.PFA}.
\begin{lemma}\label{lemma.baumgartner}
	Let $T$ be a tree and $B$ be a set cofinal branches of $T$.
	Suppose $\kappa \coloneqq \height(T)$ is regular and $|B| \leq \kappa$.
	Then there is a Baumgartner function $g: B \to T$.
\end{lemma}
\begin{proof}
	Let $\langle b_\alpha:\alpha < \mu \rangle$ enumerate $B$, with $\mu \leq \kappa$.
	Recursively define $g$ by $g(b_\alpha) \coloneqq \min ( b_\alpha - \medcup \{ b_\beta:\beta < \alpha \} )$.
	This can be done since $\kappa$ is regular.
	Suppose $g(b_\alpha) < g(b_{\alpha'})$ for some $\alpha, \alpha' < \mu$.
	Then $g(b_{\alpha'}) \in b_{\alpha'}$, so $g(b_\alpha) \in b_{\alpha'}$, so $\alpha < \alpha'$ and thus
	$g(b_{\alpha'}) \notin b_\alpha$.
\end{proof}

Recall that a tree $T$ is said to \emph{not split at limit levels} if
for all $t, t' \in T$ such that $\height t = \height t'$ is a limit ordinal and $\{ s \in T:s < t \} = \{ s \in T:s < t' \}$
it follows that $t = t'$.
\begin{lemma}\label{lemma.like_regressive}
	Let $T$ be a tree that does not split at limit levels and suppose $B$ is a set of cofinal branches of $T$.
	Suppose $g: B \to T$ is a Baumgartner function.
	Suppose $\langle \alpha_\nu:\nu < \omega_1 \rangle$ is continuous and increasing.
	Let $\alpha \coloneqq \sup_{\nu < \omega_1} \alpha_\nu$ and $t \in T_\alpha$.
	Suppose that for all $\nu < \omega_1$ there is $b_\nu \in B$ such that $g(b_\nu) < t \restriction \alpha_\nu \in b_\nu$.
	Then there is a stationary $S \subset \omega_1$ such that $b_\nu = b_{\nu}$ for all $\nu, \nu' \in S$.
	In particular there is an $s < t$ such that $t \in g^{-1}(s)$.
\end{lemma}
\begin{proof}
	For $\nu < \omega_1$ let $r(\nu) \coloneqq \min \{ \rho < \nu\ |\ \height g(b_\nu) < \alpha_\rho \}$.
	Then $r$ is regressive and thus constant on a stationary set $S \subset \omega_1$.
	As $g$ is a Baumgartner function, this implies $g$ is constant on the set $\{ b_\nu\ |\ \nu \in S \}$.
	But $g$ is injective, so $b_\nu = b_{\nu'}$ for $\nu, \nu' \in S$.
\end{proof}

\begin{definition}\label{def.covering_approximation}
	Let $V\subset W$ be a pair of transitive models of\/ \ZFC.
	\begin{itemize} 
		\item $(V,W)$ satisfies the $\mu$-covering property if the class $P_\mu^V V$ is cofinal in $P_\mu^W V$, that is, for every $x \in W$ with $x \subset V$ and $|x| < \mu$ there is $z \in P_\mu^V V$ such that $x \subset z$.
		\item $(V,W)$ satisfies the $\mu$-approximation property if for all $x \in W$, $x \subset V$, it holds that if $x \cap z \in V$ for all $z \in P_\mu^V V$, then $x \in V$.
	\end{itemize}
	A forcing $\mathbb{P}$ is said to satisfy the $\mu$-covering property or the $\mu$-approximation property if for every $V$-generic $G \subset \mathbb{P}$ the pair $(V, V[G])$ satisfies the $\mu$-covering property or the $\mu$-approximation property respectively.
\end{definition}
These properties have been introduced and extensively studied by Hamkins, see for example~\cite{hamkins}.

The following lemma is the essential argument in the proof of Theorem~\ref{theorem.PFA->ISP}.
Extracting it has the advantage that it can be applied to a wider class of different forcings, so that it can yield more information about the nature of the guessing models and $I_\IS[\omega_2, \lambda]$.
\begin{lemma}\label{lemma.PFA->ISP}
	Let $\theta$ be sufficiently large.
	Assume $\mathbb{P}$ satisfies the $\omega_1$-covering and the $\omega_1$-approximation properties and collapses $2^\lambda$ to $\omega_1$. 
	Then in $V^{\mathbb{P}}$ there is a ccc forcing $\dot{\mathbb{Q}}$ and some $w \in H_\theta$ such that
	\begin{equation*}
		\{ M \in P'_{\omega_2} H_\theta\ |\ w \in M,\ \exists G \subset \mathbb{P} * \dot{\mathbb{Q}}\ \text{$G$ is $M$-generic} \} \subset \mathcal{G}_\kappa^\lambda H_\theta,
	\end{equation*}
	and every such $M$ is internally unbounded, that is, $M \cap P_{\omega_1} M$ is cofinal in $P_{\omega_1} M$.
\end{lemma}
\begin{proof}
	Let $B \coloneqq \vphantom{2}^\lambda 2$.

	Work in $V^{\mathbb{P}}$.
	Let $\dot{c}: \omega_1 \to P_{\omega_1} \lambda$ be continuous and cofinal.
	As $\mathbb{P}$ satisfies the $\omega_1$-covering property, we may assume that $\dot{c}(\alpha + 1) \in V$ for all $\alpha < \omega_1$.
	Define
	\begin{equation*}
		\dot{T} \coloneqq \{ h \restriction \dot{c}(\alpha)\ |\ h \in B,\ \alpha < \omega_1 \}
	\end{equation*}
	As $\mathbb{P}$ satisfies the $\omega_1$-approximation property, we have that $B$ is the set of cofinal branches through $\dot{T}$.
	
	Since $|B|= \omega_1$, we can apply Lemma~\ref{lemma.baumgartner} and get a Baumgartner function $\dot{g}: B \to \dot{T}$.
	Let $\dot{l}: \omega_1 \to B$ be a bijection.
	Let
	\begin{align*}
		\dot{T}^0 & \coloneqq \{ t \in \dot{T}:\exists b \in B\ \dot{g}(b) < t \in b \},\\
		\dot{T}^1 & \coloneqq \dot{T} - \dot{T}^0.
	\end{align*}
	Note that $\dot{T}^1$ does not have cofinal branches.
       	Thus there is a ccc forcing $\dot{\mathbb{Q}}$ that specializes $\dot{T}^1$ with a specialization map $\dot{f}$.

	Now work in $V$.
	Let $w \in H_\theta$ contain all the relevant information, and let $M \in P'_{\omega_2} H_\theta$ be such that $w \in M$ and there is an $M$-generic $G_0 * G_1 \subset \mathbb{P} * \dot{\mathbb{Q}}$.

	By the usual density arguments, $c \coloneqq \dot{c}^{G_0}: \omega_1 \to P_{\omega_1} (M \cap \lambda)$ is continuous and cofinal and $c(\alpha + 1) \in M$ for all $\alpha < \omega_1$.
	Therefore $M$ is internally unbounded.
	We let $g \coloneqq \dot{g}^{G_0}$, $T \coloneqq \dot{T}^{G_0}$, $T^0 \coloneqq (\dot{T}^0)^{G_0}$, $T^1 \coloneqq (\dot{T}^1)^{G_0}$, $l \coloneqq \dot{l}^G$, and $f \coloneqq \dot{f}^{G_0 * G_1}$. 
	Define $B \restriction M \coloneqq \{ h \restriction M\ |\ h \in B \cap M\}$.
	Then we can use the facts that $G_0 * G_1$ is an $M$-generic filter and that $V^{\mathbb{P}} \models \rng \dot{l} = B$ to argue that
	\begin{itemize}
		\item $l: \omega_1 \to B \cap M$ is bijective,
		\item $T = \medcup \{h \restriction c(\alpha)\ |\ h \in B \cap M ,\ \alpha < \omega_1 \}$,
		\item $g: B \restriction M \to T$ is a Baumgartner function,\footnote{Here we naturally identify $\dom g = B \cap M$ with $B \restriction M$, which is a set of uncountable branches of $T$.}
		\item $T = T^0 \cup T^1$,
		\item $f: T^1 \to \omega$ is a specialization map.
	\end{itemize}
	
	\begin{claim}\label{claim.keyclaim1}
		$B \restriction M$ is the set of uncountable branches of $T$.
	\end{claim}
	\begin{claimproof}	
		It is clear that $B \restriction M$ is included in the set of uncountable branches of $T$.	
		For the other inclusion, observe that if $h$ is a branch through $T$, then $h$ must be a branch through $T_0$ since the specialization map $f$ witnesses that $T_1$ cannot have uncountable branches.
		This means that $h \restriction c(\alpha) \in T_0$ for eventually all $\alpha$.
		So for each such $\alpha$ there is a unique $b_\alpha \in B \restriction M$ such that $g(b_\alpha) \subset h \restriction c(\alpha) \subset b_\alpha$.
		Thus for eventually all $\alpha < \omega_1$ we have $\dom g(b_\alpha) = c(\beta_\alpha)$ for some $\beta_\alpha < \alpha$, and we may assume that there is a $\beta < \omega_1$ such that $\beta_\alpha = \beta$ for stationarily many $\alpha < \omega_1$.
		Hence if $\alpha$ is such that $\beta_\alpha = \beta$, then $h = b_\alpha \in B \restriction M$.
	\end{claimproof}

	\begin{claim}\label{claim.keyclaim2} 
		$t \in B \restriction M$ if and only if $t$ is the characteristic function of $d \cap M$ for some $M$-approximated $d \subset \lambda$.		
	  \end{claim}
	  \begin{claimproof}
	  	If $t \in B \restriction M$, then $t = h \restriction M$ for some $h \in B \cap M$, and $h$ is the characteristic function of some $d \in M \cap P\lambda$.

		 For the other direction pick an $M$-approximated $d \subset \lambda$, and let $t$ be the characteristic function of $d\cap M$.
		 We claim that $t$ is a branch through $T$ and thus in $B \restriction M$ by Claim~\ref{claim.keyclaim1}.
		 To see this observe that $c(\alpha + 1) \in M$ for all $\alpha < \omega_1$, so that $t \restriction c(\alpha + 1)$ is the characteristic function of $d \cap c(\alpha + 1)$, which is in $M$ since $d$ is $M$-approximated.
		 Thus $t \restriction c(\alpha + 1) \in T$.
	\end{claimproof}
	To see $M$ is $\lambda$-guessing, let $d \subset \lambda$ be $M$-approximated.
	Then by Claim~\ref{claim.keyclaim2} the characteristic function $t$ of $d \cap M$ is in $B \restriction M$.
	So there is $h \in B \cap M$ such that $t = h \restriction M$.
	Let $e \in M$ be such that $h$ is its characteristic function.
	Then $e \cap M = d \cap M$, and we are done.
\end{proof}

To apply Lemma~\ref{lemma.PFA->ISP}, we need an appropriate forcing.
The simplest and earliest example comes from~\cite{mitchell}.
We let $\mathbb{C}$ denote the forcing for adding a Cohen real.
See~\cite{krueger} for a proof of the following theorem.
\begin{theorem}\label{theorem.club_through_E}
	Let $\gamma \geq \omega_1$.
	Then the forcing $\mathbb{C} * \Coll(\omega_1, \gamma)$ is proper and satisfies the $\omega_1$-approximation property.
\end{theorem}

\begin{theorem}\label{theorem.PFA->ISP}
	\PFA\ implies $\ISP(\omega_2)$ holds.
\end{theorem}
\begin{proof}
	Let $\theta$ be large enough, $\lambda \geq \omega_2$, and $\mathbb{P} \coloneqq \mathbb{C} * \Coll(\omega_1, 2^\lambda)$.
	Then $\mathbb{P}$ is proper and satisfies the $\omega_1$-approximation property by Theorem~\ref{theorem.club_through_E}.
	Thus by Lemmas~\ref{lemma.PFA_stationarily_often} and~\ref{lemma.PFA->ISP} the set $\mathcal{G}_{\omega_2}^\lambda H_\theta$ is stationary in $P_{\omega_2} H_\theta$.
	Therefore by Proposition~\ref{prop.guessing->ISP} we can conclude that $\ISP(\omega_2, \lambda)$ holds.
\end{proof}

Krueger \cite{krueger.IC,krueger.IA} has shown there is a great variety of forcings $\dot{\mathbb{P}}$ living in $V^{\mathbb{C}}$ such that $\mathbb{C}*\dot{\mathbb{P}}$ has the $\omega_1$-approximation and the $\omega_1$-covering properties.
These forcings can be used to show that under \PFA, there are stationarily many guessing models that are internally club.
As guessing models are not internally approachable, this gives another separation of the properties internally club and internally approachable.
Under \MM, one can use these forcings to show there are stationarily many guessing models that are internally unbounded but not internally stationary and also stationarily many that are internally stationary but not internally club, see also~\cite{VIA10}.

Strullu~\cite{STR10} has shown the principle $\ITP(\omega_2)$ follows from $\MRP + \MA$, where \MRP\ is the mapping reflection principle introduced by Moore~\cite{moore.MRP}.

It is furthermore worth noting that unlike $\ISP(\omega_2)$, the principle $\ITP(\omega_2)$ can already be proved by applying \PFA\ to a forcing of the form $\sigma$-closed $*$ ccc, see~\cite{diss}.

The next corollary is originally independently due to Foreman and Todor\v{c}evi\'c, see~\cite{koenig}.
\begin{corollary}\label{cor.PFA->nonAP}
	\PFA\ implies the approachability property fails for $\omega_1$, that is, $\omega_2 \notin I[\omega_2]$, where $I[\omega_2]$ denotes the approachability ideal on $\omega_2$.
\end{corollary}
\begin{proof}
	It is not hard to see that $I[\omega_2] \subset I_\IS[\omega_2, \omega_2]$.
\end{proof}

The failure of various square principles under \PFA\ is originally due to Todor\-\v{c}evi\'c and Magidor, see~\cite{todorcevic.note_on_PFA} and~\cite[Theorem~6.3]{schimmerling}.
See~\cite{weiss} for the notation used in Corollary~\ref{cor.PFA->non_square}.
\begin{corollary}\label{cor.PFA->non_square}
	Suppose \PFA\ holds and $\cf \lambda \geq \omega_2$.
	Then $\non \square_{\cof(\omega_1)}(\omega_2, \lambda)$.
\end{corollary}
\begin{proof}
	This follow from Theorem~\ref{theorem.PFA->ISP} and~\cite[Theorem~4.2]{weiss}.
\end{proof}

\section{An interlude on forcing}

\begin{definition}
	Let $\mathbb{P}$ be a forcing.
	We say $\mathbb{P}$ is a \emph{standard iteration of length $\kappa$} if
	\begin{enumerate}[label=(\roman*)]
		\item $\mathbb{P}$ is the direct limit of an iteration $\langle \mathbb{P}_\alpha\ |\ \alpha < \kappa \rangle$ that takes direct limits stationarily often,
		\item $\mathbb{P}_\alpha$ has size less than $\kappa$ for all $\alpha<\kappa$.
	\end{enumerate}
\end{definition}

It is a classical result that the $\mu$-cc is preserved by iterations of length $\mu$ of posets of size less than $\mu$ that take direct limits stationarily often.
So the following lemma does not come as a surprise but nonetheless has not been observed so far.
\begin{lemma}\label{keylem1}
	Let $\mathbb{P}$ be a standard iteration of length $\kappa$.
	Then $\mathbb{P}$ is $\kappa$-cc and satisfies the $\kappa$-approximation property.
\end{lemma}
\begin{proof}
	Let $\mathbb{P}$ be the direct limit of $\langle \mathbb{P}_\alpha\ |\ \alpha < \kappa \rangle$.
	It suffices to verify the $\kappa$-approximation property for subsets of ordinals.
	The proof is by induction on $\lambda\geq\kappa$.

	We start with the proof of the base case $\lambda = \kappa$.
	We need to show that if $p \in \mathbb{P}$ and $\dot{h} \in V^{\mathbb{P}}$ are such that $p \forces_{\mathbb{P}} \dot{h} \in \vphantom{2}^\kappa 2$ and $p \forces_{\mathbb{P}} \forall \alpha < \kappa\ \dot{h} \restriction \alpha \in V$, then $p \forces_{\mathbb{P}} \dot{h} \in V$.
	So assume to the contrary there is $\bar{p} \leq p$ such that $\bar{p} \forces_{\mathbb{P}} \dot{h} \notin V$.

	Let $P=\{p_\xi\ |\ \xi<\kappa\}$ and let $C_0$ be the club of all $\alpha < \kappa$ such that $\medcup \{ \mathbb{P}_\xi\ |\ \xi < \alpha \} = \{p_\xi\ |\ \xi < \alpha \}$.
	Define $S \coloneqq \{ \alpha < \kappa\ |\ \mathbb{P}_\alpha\ \text{is direct limit} \}$.
	$S$ is stationary by assumption, and if $\alpha \in S \cap C_0$, then $\mathbb{P}_\alpha = \{ p_\xi\ |\ \xi < \alpha \}$.

	For $\xi < \kappa$ let $A_\xi \subset \mathbb{P}$ be a maximal antichain below $\bar{p}$ that decides the value of $\dot{h}(\xi)$.
	Then $C \coloneqq \{ \alpha \in C_0\ |\ \forall \xi < \alpha\ A_\xi \subset \mathbb{P}_\alpha \}$ is club.
	For $\alpha \in C$ let
	\begin{equation*}
		\dot{h}_\alpha \coloneqq \{ \langle (\xi,i) , p \rangle\ |\ \xi < \alpha,\ p \in \mathbb{P}_\alpha,\ p \forces_{\mathbb{P}} \dot{h}(\xi)=i \}.
	\end{equation*}
	Then $\dot{h}_\alpha \in V^{\mathbb{P}_\alpha}$ and $\bar{p} \forces_{\mathbb{P}} \dot{h}_\alpha \in \vphantom{2}^\alpha 2$.

	\begin{claim}\label{keyfa0}
		$\bar{p} \forces_{\mathbb{P}} \dot{h} \restriction \alpha = \dot{h}_\alpha$ for all $\alpha \in C$.
	\end{claim}
	\begin{claimproof}
		Suppose to the contrary that for some $\alpha \in C$ there are $q \leq \bar{p}$ and $\xi < \alpha$ such that $q \forces_{\mathbb{P}} \dot{h}(\xi) \neq \dot{h}_\alpha(\xi)$.
		Let $r \in A_\xi$ be compatible with $q$.
		Then $r \forces_{\mathbb{P}} \dot{h}(\xi) = i$ for some $i < 2$.
		But as $A_\xi \subset \mathbb{P}_\alpha$, this also means $r \forces_{\mathbb{P}} \dot{h}_\alpha(\xi) = i$, contradicting its compatibility with $q$.
	\end{claimproof}

	\begin{claim}\label{keyfa1}
		$\bar{p} \forces_{\mathbb{P}_\alpha} \dot{h}_\alpha \in V$ for all $\alpha \in C$.
	\end{claim}
	\begin{claimproof}
		Assume towards a contradiction that some for some $q \leq \bar{p}$ and $\alpha \in C$ we have $ q \forces_{\mathbb{P}_\alpha} \dot{h}_\alpha \notin V$.
		Then for each $g\in \vphantom{2}^\alpha 2$ there is a maximal antichain
		$A_g$ among the conditions in $\mathbb{P}_\alpha$ below $q$ such that for any element $r \in A_g$, there is $\xi_r < \alpha$ such that $r \forces_{\mathbb{P}_\alpha} \dot{h}_\alpha(\xi_r) \neq g(\xi_r)$.
		This means that any $\langle(\xi_r,i),p\rangle\in\dot{h}_\alpha$ such that $p$ is compatible with $r$ is such that $g(\xi_r)\neq i$. 
		This in turn means that $r\forces_{\mathbb{P}} \dot{h}_\alpha(\xi_r) \neq g(\xi_r)$ for any $r\in A_g$ and for any $g\in \vphantom{2}^\alpha 2$. 

		Since a maximal antichain in $\mathbb{P}_\alpha$ is also a maximal antichain in $\mathbb{P}$, this implies that $q \forces_{\mathbb{P}} \dot{h}_\alpha \notin V$,
		which is impossible by Claim~\ref{keyfa0}.
	\end{claimproof}

	For $\alpha \in S \cap C_0$ by Claim~\ref{keyfa1} $\bar{p} \forces_{\mathbb{P}_\alpha} \dot{h}_\alpha \in V$, so there are $p_\xi \in \mathbb{P}_\alpha$, $p_\xi \leq \bar{p}$, and $g_\alpha \in \vphantom{2}^\alpha 2$ such that $p_\xi \forces_{\mathbb{P}_\alpha} \dot{h}_\alpha = g_\alpha$.
	Since $\alpha \in S \cap C_0$, we have $\xi < \alpha$, so for some stationary $S_0 \subset S \cap C_0$ we may assume $\xi$ is fixed.
	But then $p_\xi \forces_{\mathbb{P}_\alpha} \dot{h} \restriction \alpha = \dot{h}_\alpha = g_\alpha$ for all $\alpha \in S_0$,
	so that $p_\xi \forces_{\mathbb{P}} \dot{h} = \medcup_{\alpha \in S_0} \dot{h}_\alpha = \medcup_{\alpha \in S_0} g_\alpha \in V$, contradicting $p_\xi \leq \bar{p}$.

	Now we prove the lemma for $\lambda > \kappa$, assuming it has been shown for all $\gamma < \lambda$.
	Let $p \in \mathbb{P}$ and $\dot{h} \in V^{\mathbb{P}}$ be such that $p \forces_{\mathbb{P}} \dot{h} \in \vphantom{2}^\lambda 2$ and
	$p \forces_{\mathbb{P}} \forall z \in P^V_\kappa V\ \dot{h} \restriction z \in V$.

	First suppose $\cf \lambda > \kappa$.
	By the induction hypothesis we know that $p \forces_{\mathbb{P}} \forall \gamma < \lambda\ \dot{h} \restriction \gamma \in V$.
	For every $\gamma < \lambda$ there is $\alpha_\gamma < \kappa$ and $g_\gamma \in \vphantom{2}^\gamma 2$ such that $p_{\alpha_\gamma} < p$ and $p_{\alpha_\gamma} \forces \dot{h} \restriction \gamma = g_\gamma$.
	Thus there is an unbounded $U \subset \lambda$ such that $\alpha_\gamma = \alpha_{\gamma'}$ for all $\gamma, \gamma' \in U$,
	so that for $\gamma \in U$ we have $p_{\alpha_\gamma} \forces \dot{h} = \medcup_{\gamma \in U} g_\gamma \in V$.

	If $\cf \lambda \leq \kappa$, let $U \subset \lambda$ be cofinal of order type $\cf \lambda$, and set
	\begin{equation*}
		T \coloneqq \{ g \in \vphantom{2}^{<\lambda} 2\ |\  \exists q \leq p\ \exists \gamma \in U\ q \forces_{\mathbb{P}} \dot{h} \restriction \gamma = g \}.
	\end{equation*}
	Then $T$, ordered by end extension, is a tree of height $\cf \lambda$.
	As $\mathbb{P}$ is $\kappa$-cc, all levels of $T$ have size less than $\kappa$.
	Let $X$ be a set of size at most $\kappa$ such that for every pair of incompatible elements $g, g' \in T$ there is $\alpha \in X$ such that $g(\alpha) \neq g'(\alpha)$.
	By the induction hypothesis we have $p \forces_{\mathbb{P}} \dot{h} \restriction X \in V$.
	But $p \forces_{\mathbb{P}} \dot{h} = \medcup \{ g \in T\ |\ g \restriction X = \dot{h} \restriction X \}$,
	so that $p \forces_{\mathbb{P}} \dot{h} \in V$.
\end{proof}

\section{The principles {\sffamily TP} and {\sffamily ITP} in generic extensions}

\begin{lemma}\label{lemma.thin_from_groundmodel}
	Let $V \subset W$ be a pair of models of\/ \ZFC\ that satisfies the $\kappa$-covering property, and suppose $\kappa$ is inaccessible in $V$.
	Suppose $D = \langle d_a\ |\ a \in P^W_\kappa \lambda \rangle$ is a $P^W_\kappa \lambda$-list such that for every $a \in P^W_\kappa \lambda$ there is $z_a \in V$ such that $d_a = z_a \cap a$.
	Then $D$ is thin.
\end{lemma}
\begin{proof}
	Work in $W$.
	Let $c \in P_\kappa \lambda$.
	By the $\kappa$-covering property there is $\bar{c} \in P^V_\kappa \lambda$ such that $c \subset \bar{c}$.
	Also we have $\{ d_a \cap c\ |\ c \subset a \in P^W_\kappa \lambda \} = \{ z_a \cap \bar{c} \cap c\ |\ c \subset a \in P^V_\kappa \lambda \} \subset \{ z \cap c\ |\ z \in P^V \bar{c} \}$.
	But the latter set has cardinality less than than $\kappa$ since $\kappa$ is inaccessible in $V$.
\end{proof}

\begin{proposition}\label{prop.old_I_subset_of_new_I}
	Let $V \subset W$ be a pair of models of\/ \ZFC\ that satisfies the $\kappa$-covering and the $\kappa$-approximation properties, and suppose $\kappa$ is inaccessible in $V$.
	Then
	\begin{equation*}
		I_\IT^V[\kappa, \lambda] \subset I^W_\IT[\kappa, \lambda].
	\end{equation*}
\end{proposition}
\begin{proof}
	Work in $W$.
	For $A \in I_\IT^V[\kappa, \lambda]$ let $\langle d_a\ |\ a \in P^V_\kappa \lambda \rangle \in V$ be $A$-effable in $V$.

	Then by Lemma~\ref{lemma.thin_from_groundmodel} $\langle d_a\ |\ a \in P_\kappa \lambda \rangle$ is thin, where $d_a \coloneqq \emptyset$ for $a \notin V$.

	Suppose $\langle d_a\ |\ a \in P_\kappa \lambda \rangle$ were not $A$-effable.
	Let $S \subset A$ be stationary and $d \subset \lambda$ such that $d_x = d \cap x$ for all $x \in S$.
	Suppose $d \notin V$.
	Then, by $\kappa$-approximation property, there is a $z \in P^V_\kappa \lambda$ such that $d \cap z \notin V$.
	But for $x \in S$ with $z \subset x$ we have $d \cap z = d \cap x \cap z = d_x \cap z \in V$, a contradiction.
	Therefore $d \in V$, and $S \subset \bar{S} \coloneqq \{ x \in P^V_\kappa \lambda\ |\ d_x = d \cap x \} \in V$.
	Since $\langle d_a\ |\ a \in P^V_\kappa \lambda \rangle \in V$ is $A$-effable in $V$, $\bar{S}$ is not stationary in $V$.
	So there exists $C \in V$, $C \subset P^V_\kappa \lambda$ club in $V$ such that $C \cap \bar{S} = \emptyset$.
	Let $f: P_\omega \lambda \to P_\kappa \lambda$ be in $V$ such that $\Cl_f^V \subset C$.
	But then, by the stationarity of $S$, there is an $x \in S$ such that $x \in \Cl_f$, so that $x \in C \cap \bar{S}$,
	a contradiction.
\end{proof}

\begin{theorem}\label{theorem.old_filter_subset_of_new_filter}
	Let $V \subset W$ be a pair of models of\/ \ZFC\ that satisfies the $\kappa$-covering property and the $\tau$-approximation property for some $\tau < \kappa$, and suppose $\kappa$ is inaccessible in $V$.
	Then
	\begin{equation*}
		P^W_\kappa \lambda - P^V_\kappa \lambda \in I^W_\IT[\kappa, \lambda],
	\end{equation*}
	which furthermore implies
	\begin{equation*}
		F_\IT^V[\kappa, \lambda] \subset F^W_\IT[\kappa, \lambda].
	\end{equation*}
	So in particular, if $W \models \ITP(\kappa, \lambda)$, then $V \models \ITP(\kappa, \lambda)$.
\end{theorem}
\begin{proof}
	Work in $W$.
	Let $B \coloneqq P_\kappa \lambda - P^V_\kappa \lambda$.
	For $x \in B$ let $a_x \in P^V_\tau \lambda$ be such that $x \cap a_x \notin V$, which exists by the $\tau$-approximation property.
	Put $d_x \coloneqq a_x \cap x$.
	For $x \in P_\kappa \lambda - B$, let $d_x \coloneqq \emptyset$.
	Then $\langle d_x\ |\ x \in P_\kappa \lambda \rangle$ is thin by Lemma~\ref{lemma.thin_from_groundmodel}.

	Suppose $\langle d_x\ |\ x \in P_\kappa \lambda \rangle$ were not $B$-effable.
	Then there are $d \subset \lambda$ and $U \subset B$ be such that $U$ is cofinal and $d_x = d \cap x$ for all $x \in U$.
	Define a $\subset$-increasing sequence $\langle x_\alpha\ |\ \alpha < \tau^+ \rangle$ with $x_\alpha \in U$ for
	all $\alpha < \tau^+$ and a sequence $\langle e_\alpha\ |\ \alpha < \tau^+ \rangle$ such that
	$x_\alpha \subset e_\alpha$ and $e_\alpha \in P^V_\kappa \lambda$ for all $\alpha < \tau^+$ as follows.
	Let $\beta < \tau^+$ and suppose $\langle x_\alpha\ |\ \alpha < \beta \rangle$ and
	$\langle e_\alpha\ |\ \alpha < \beta \rangle$ have been defined.
	Let $x_\beta \in U$ be such that $\medcup_{\alpha < \beta} ( x_\alpha \cup a_\alpha \cup e_\alpha ) \subset x_\beta$,
	and let $e_\beta \in P^V_\kappa \lambda$ be such that $x_\beta \subset e_\beta$, which exists by the $\kappa$-covering property.

	Then $\langle d_{x_\alpha}\ |\ \alpha < \tau^+ \rangle$ is $\subset$-increasing as $d_{x_\alpha} = d \cap x_\alpha$ for all $\alpha < \tau^+$,
	and since $| d_{x_\alpha} | < \tau$ for all $\alpha < \tau^+$,
	there is $\gamma < \tau^+$ such that $d_{x_\alpha} = d_{x_{\alpha'}}$ for all $\alpha, \alpha' \in [\gamma, \tau^+)$.
	But then
	$a_{x_{\gamma+1}} \cap e_\gamma \subset a_{x_{\gamma+1}} \cap x_{\gamma+1} = d_{x_{\gamma+1}} = d_{x_\gamma} \subset e_{\gamma}$
	and $d_{x_{\gamma + 1}} \subset a_{x_{\gamma + 1}}$, so that $d_{x_\gamma} = a_{x_{\gamma+1}} \cap e_\gamma \in V$, a contradiction.

	To see $F_\IT^V[\kappa, \lambda] \subset F^W_\IT[\kappa, \lambda]$, let $A \in F^V_\IT[\kappa, \lambda]$.
	Then $P^V_\kappa \lambda - A \in I^V_\IT[\kappa, \lambda]$, so by Proposition~\ref{prop.old_I_subset_of_new_I}
	$P^V_\kappa \lambda - A \in I^W_\IT[\kappa, \lambda]$.
	Thus $P^W_\kappa \lambda - A = ( P^W_\kappa \lambda - P^V_\kappa \lambda ) \cup
	( P^V_\kappa \lambda - A ) \in I^W_\IT[\kappa, \lambda]$, which means $A \in F^W_\IT[\kappa, \lambda]$.
\end{proof}
Note that by~\cite[Theorem~1.1]{gitik.nonsplitting} the set $P^W_\kappa \lambda - P^V_\kappa \lambda$ in Theorem~\ref{theorem.old_filter_subset_of_new_filter} is stationary for $\lambda \geq \kappa^+$ if there is a real in $W - V$.
We will now weaken the assumption that $(V, W)$ satisfies the $\tau$-approximation property for some $\tau < \kappa$ to the $\kappa$-approximation property, so that this kind of argument can be exploited for a wider range of forcing constructions.

\begin{theorem}\label{theorem.pull_back_TP}
	Let $V \subset W$ be a pair of models of\/ \ZFC\ that satisfies the $\kappa$-covering and the $\kappa$-approximation properties, and suppose $\kappa$ is inaccessible in $V$.
	If $W \models \TP(\kappa, \lambda)$, then $V \models \TP(\kappa, \lambda)$.
\end{theorem}
\begin{proof}
	In $V$, let $D = \langle d_a\ |\ a \in P_\kappa \lambda \rangle$ be a $P_\kappa \lambda$-list.

	Now work in $W$. For every $a \in P_\kappa \lambda$ let, by the $\kappa$-covering property, $z_a \in P^V_\kappa \lambda$ be such that $a \subset z_a$.
	Define a $P_\kappa \lambda$-list $E = \langle e_a\ |\ a \in P_\kappa \lambda \rangle$ by $e_a \coloneqq d_{z_a} \cap a$.
	Then $E$ is thin by Lemma~\ref{lemma.thin_from_groundmodel}.

	Thus by $\TP(\kappa,\lambda)$ there is a cofinal branch $d$ for $E$.
	So for all $y \in P_\kappa \lambda$ there is $a \in P_\kappa \lambda$, $y \subset a$, such that $e_a \cap y = d\cap y$.
	In particular
	\begin{equation*}
		d \cap y = e_a \cap y = d_{z_a} \cap a \cap y = d_{z_a} \cap y.
	\end{equation*}
	Thus if $y \in P^V_\kappa \lambda$, then $d \cap y \in V$, so that $d \in V$ by the $\kappa$-approximation property.
	This means $d \in V$.
	But $d$ is also a cofinal branch for $D$ in $V$.
\end{proof}

\begin{corollary}\label{cor.strongly_compact_in_groundmodel}
	Let $\mathbb{P}$ be a standard iteration of length $\kappa$ and suppose $\kappa$ is inaccessible.
	If\/ $\mathbb{P}$ forces $\TP(\kappa)$, then $\kappa$ is strongly compact.
\end{corollary}
\begin{proof}
	This follows directly from Lemma~\ref{keylem1} and Theorem~\ref{theorem.pull_back_TP}.
\end{proof}
Notice that, together with Theorem~\ref{theorem.PFA->ISP}, Corollary~\ref{cor.strongly_compact_in_groundmodel} implies the following remarkable corollary.
\begin{corollary}\label{cor.PFA_requires_strongly_compact}
	Suppose $\kappa$ is inaccessible and \PFA\ is forced by a standard iteration of length $\kappa$ that collapses $\kappa$ to $\omega_2$.
	Then $\kappa$ is strongly compact.
\end{corollary}
Corollary~\ref{cor.PFA_requires_strongly_compact} says that any of the known methods for producing a model of \PFA\ from a large cardinal assumption requires at least a strongly compact cardinal.
This can be improved to the optimal result if we require the iteration for forcing \PFA\ to be proper.
For this purpose we introduce an ad-hoc definition.
\begin{definition}
	Let $V \subset W$ be a pair of models of\/ \ZFC\ that satisfies the $\kappa$-covering and the $\kappa$-approximation properties, and suppose $\kappa$ is inaccessible in $V$.
	We say $M \in (P_\kappa' H_\theta^V)^W$ is \emph{$V$-guessing} if for all $z \in M$ and all $d \in P^V z$ there is an $e \in M$ such that $d \cap M = e \cap M$.
\end{definition}

The following two propositions should be seen as analogs of Propositions~\ref{prop.guessingmodels} and~\ref{prop.guessing->ISP}.
\begin{proposition}\label{prop.V-guessing}
	Let $V \subset W$ be a pair of models of\/ \ZFC\ that satisfies the $\kappa$-covering and the $\kappa$-approximation properties, and suppose $\kappa$ is inaccessible in $V$.
	Assume $W \models \ITP(\kappa, |H_\theta^V|)$ for some large enough $\theta$.
	Then in $W$ the set
	\begin{equation*}
		\{ M \in P'_\kappa H_\theta^V\ |\ \text{$M$ is $V$-guessing and closed under countable suprema} \}
	\end{equation*}
	is stationary.\footnote{However, it need not be a subset of $V$.}
\end{proposition}
\begin{proof} 
	Work in $W$.
	By~\cite[Theorem~3.5]{weiss}, we have that the set of all $M \in P_\kappa' H_\theta^V$ that are closed under countable suprema belongs to $F_\IT[\kappa, H_\theta^V]$.
	Assume that there were a set $A \notin I_\IT[\kappa, H_\theta^V]$ such that for all $M \in A$ there is $z_M \in M$ and $d_M \in P^V z_M$ such that $d_M \cap M \neq e \cap M$ for all $e \in M$.
	Then $D \coloneqq \langle d_M \cap M\ |\ M \in A \rangle$ is thin by Lemma~\ref{lemma.thin_from_groundmodel}.
	Thus by $\ITP(\kappa, |H_\theta^V|)$ there is an ineffable branch $d$ for $D$, and by the $\kappa$-approximation property we have $d \in V$.
	Let $S \coloneqq \{ M \in A\ |\ d_M \cap M = d \cap M \}$.
	Then $S \in V$ is stationary, and we may assume $z_M = z$ for some $z \in H_\theta^V$ and all $M \in S$.
	As $P^V z \subset H_\theta^V$ and $d \subset z$, there is an $M \in S$ such that $d \in M$, a contradiction.
\end{proof}

\begin{theorem}\label{theorem.pull_back_ISP}
	Let $V \subset W$ be a pair of models of\/ \ZFC\ that satisfies the $\kappa$-covering and the $\kappa$-approximation properties.
	Let $\kappa$ be inaccessible in $V$ and $\lambda$ be regular in $W$.
	Suppose that for all $\gamma < \kappa$ and every $S \subset \cof(\omega) \cap \gamma$ in $V$ it holds that $V \models$ ``$S$ is stationary in $\gamma$'' if and only if $W \models$ ``$S$ is stationary in $\gamma$.''
	Let $\theta$ be large enough.
	Suppose $M \in (P_\kappa' H_\theta^V)^W$ is a $V$-guessing model closed under countable suprema such that $\lambda \in M$.
	Then $M \cap \lambda \in V$ and $V \models \ITP(\kappa,\lambda)$.
\end{theorem}
\begin{proof}
	Let $\langle S_\alpha\ |\ \alpha < \lambda \rangle \in M$ be a partition of $\cof(\omega) \cap \lambda$ into sets stationary in $V$.
	Let $\lambda_M \coloneqq \sup (M \cap \lambda)$.
	\begin{claim}\label{claim.magidor}
		It holds that
		\begin{equation*}
			M \cap \lambda = \{ \delta < \lambda\ |\ V \models \text{$S_\delta$ is stationary in $\lambda_M$} \} \in V.
		\end{equation*}
	\end{claim}
	\begin{claimproof}
		For one direction, let $\delta$ be such that $V \models$ ``$S_\delta$ is stationary in $\lambda_M$.''
		Notice that $\cf^V \lambda_M < \kappa$, so $W \models$ ``$S_\delta$ is stationary in $\lambda_M$.''
		As $M$ is closed under countable suprema, we get that $S_\delta \cap M \neq \emptyset$.
		Thus if $\beta \in S_\delta \cap M$, then $\delta$ is definable in $M$ as the $\alpha$ for which $\beta \in S_\alpha$, so that $\delta \in M$.

		For the other direction, let $\delta \in M \cap \lambda$ and let $C \in V$ be club in $\lambda_M$.
		As $C \subset \lambda \in M$ and $M$ is $V$-guessing, $C \cap M = e \cap M$ for some $e \in M$.
		Since $C \cap M$ is closed under countable suprema, $M\models$ ``$e$ is closed under countable suprema.''
		Thus $M \models e \cap S_\delta \neq \emptyset$, which proves $C \cap S_\delta \neq \emptyset$ as $e \cap S_\delta \cap M \subset C \cap S_\delta$.
	\end{claimproof}

	Now to argue that $V \models \ITP(\kappa, \lambda)$, it is enough to check that $H_\theta^V \models \ITP(\kappa, \lambda)$.
	Since $M \prec H_\theta^V$, it in turn suffices to verify $M \models \ITP(\kappa, \lambda)$.
	So let $D\in M$ be a $P^V_\kappa\lambda$-list.
	Since $M$ is $V$-guessing, $d_{M \cap \lambda} \in V$, and $d_{M \cap \lambda} \subset \lambda \in M$, we get that $d_{M \cap \lambda} = e \cap M$ for some $e \in M$.
	Then $M \models$ ``$e$ is an ineffable branch for $D$.''
\end{proof}

\begin{corollary}\label{cor.supercompact_in_groundmodel}
	Let $\mathbb{P}$ be a proper standard iteration of length $\kappa$ and suppose $\kappa$ is inaccessible.
	If\/ $\mathbb{P}$ forces $\ITP(\kappa)$, then $\kappa$ is supercompact.
\end{corollary}
\begin{proof}
	This follows from Lemma~\ref{keylem1}, Proposition~\ref{prop.V-guessing}, and Theorem~\ref{theorem.pull_back_ISP}.
\end{proof}

Under the additional premise of properness, Corollary~\ref{cor.supercompact_in_groundmodel} implies the following strongest possible version of Corollary~\ref{cor.PFA_requires_strongly_compact}.
\begin{corollary}\label{cor.proper_requires_supercompact}
	Suppose $\kappa$ is inaccessible and \PFA\ is forced by a proper standard iteration of length $\kappa$ that collapses $\kappa$ to $\omega_2$.
	Then $\kappa$ is supercompact.
\end{corollary}

It should be noted that Sakai has pointed out a serious obstruction in removing the assumption of $\mathbb{P}$ being proper in Corollary~\ref{cor.proper_requires_supercompact}.
\begin{theorem}[Sakai, 2010]\label{theorem.sakai}
	Let $\kappa$ be a supercompact cardinal, $\theta > \kappa$ be sufficiently large, and suppose there is a Woodin cardinal $\mu > \theta$.
	Suppose $W$ is the standard semiproper forcing extension such that $W \models \MM + \kappa = \omega_2$.
	Then in $W$ it holds that for every stationary preserving forcing $\mathbb{P}$ the set
	\begin{equation*}
		\{ M \in P_{\omega_2} H_\theta\ |\ \exists G \subset \mathbb{P}\ \text{$G$ is $M$-generic},\ M \cap \omega_3 \notin V \}
	\end{equation*}
	is stationary in $P_{\omega_2} H_\theta$.
\end{theorem}
In the setting of Theorem~\ref{theorem.sakai}, if one carries out the proof of Theorem~\ref{theorem.PFA->ISP} in $W$, one gets that $P_{\kappa}^W \lambda - P_{\kappa}^V \lambda \notin I_\IT^W[\kappa, \lambda]$ for $\lambda$ such that $\kappa < \lambda$ and $2^\lambda < \theta$.
This should be contrasted with Theorem~\ref{theorem.old_filter_subset_of_new_filter}.

\section{Conclusion}

There are several open problems which the results presented suggest.
The most appealing deals with the construction of an inner model in which $\omega_2$ has an arbitrary degree of supercompactness starting from a universe of sets in which $\MM$ holds.
It seems plausible to conjecture that if $\ISP(\kappa)$ holds, then for each $\lambda$ there is a simply definable transitive class in which $\kappa$ is $\lambda$-supercompact.
Such a line of thought has already been pursued by Foreman~\cite{FOR10}, where he proved that a certain strong form of Chang's conjecture for a small cardinal $\kappa$ implies that there is an $X$ such that $\kappa$ is huge in $L[X]$.
It has yet to be understood to what extent Foreman's ideas can be applied to the results of this paper;
a key issue in this context appears to be a thorough study of the properties of guessing models and of the ideals $I_\IS[\omega_2,\lambda]$ in models of \MM.

We also expect that many of the known consequences of \PFA\ and supercompactness might be obtained directly from the principle \ISP.
Examples are given in \cite{weiss}, where it is shown that $\ITP(\omega_2)$ implies the failure of some of the weakest forms of square incompatible with $\PFA$, and in \cite{VIA10}, where, using properties of guessing models, a new proof that \PFA\ implies \SCH\ is provided.
On the other hand we conjecture that $\ISP(\omega_2)$ does not decide the size of the continuum.

\ifthenelse{\boolean{usemicrotype}}{\microtypesetup{spacing=false}}{}
\bibliographystyle{amsplain}
\bibliography{megateorema}

\providecommand{\bysame}{\leavevmode\hbox to3em{\hrulefill}\thinspace}
\providecommand{\MR}{\relax\ifhmode\unskip\space\fi MR }
\providecommand{\MRhref}[2]{%
  \href{http://www.ams.org/mathscinet-getitem?mr=#1}{#2}
}
\providecommand{\href}[2]{#2}
\begin{thebibliography}{10}

\bibitem{baumgartner.PFA}
J.~E. Baumgartner, \emph{Applications of the proper forcing axiom}, Handbook of
  set-the\-o\-re\-tic to\-po\-lo\-gy, North-Holland, Amsterdam, 1984,
  pp.~913--959. \MR{776640}

\bibitem{devlin}
K.~J. Devlin, \emph{The {Y}orkshireman's guide to proper forcing}, Surveys in
  set theory, London Math. Soc. Lecture Note Ser., vol.~87, Cambridge Univ.
  Press, Cambridge, 1983, pp.~60--115. \MR{823776}

\bibitem{farah}
I.~Farah, \emph{All automorphisms of the {C}alkin algebra are inner}, Ann. of
  Math. (2), forthcoming.

\bibitem{FOR10}
M.~Foreman, \emph{Smoke and mirrors: combinatorial properties of small
  cardinals equiconsistent with huge cardinals}, Adv. Math. \textbf{222}
  (2009), no.~2, 565--595. \MR{2538021}

\bibitem{foreman_magidor_shelah}
M.~Foreman, M.~Magidor, and S.~Shelah, \emph{Martin's maximum, saturated
  ideals, and nonregular ultrafilters. {I}}, Ann. of Math. (2) \textbf{127}
  (1988), no.~1, 1--47. \MR{924672}

\bibitem{gitik.nonsplitting}
M.~Gitik, \emph{Nonsplitting subset of {${\mathcal P}\sb \kappa(\kappa\sp
  +)$}}, J. Symbolic Logic \textbf{50} (1985), no.~4, 881--894 (1986).
  \MR{820120}

\bibitem{hamkins}
J.~D. Hamkins, \emph{Gap forcing}, Israel J. Math. \textbf{125} (2001),
  237--252. \MR{1853813}

\bibitem{jensen.schimmerling.schindler.steel}
R.~Jensen, E.~Schimmerling, R.~Schindler, and J.~Steel, \emph{Stacking mice},
  J. Symbolic Logic \textbf{74} (2009), no.~1, 315--335. \MR{2499432}

\bibitem{koenig}
B.~K{\"o}nig and Y.~Yoshinobu, \emph{Fragments of {M}artin's maximum in generic
  extensions}, MLQ Math. Log. Q. \textbf{50} (2004), no.~3, 297--302.
  \MR{2050172}

\bibitem{krueger.IC}
J.~Krueger, \emph{Internally club and approachable}, Adv. Math. \textbf{213}
  (2007), no.~2, 734--740. \MR{2332607}

\bibitem{krueger}
\bysame, \emph{A general {M}itchell style iteration}, MLQ Math. Log. Q.
  \textbf{54} (2008), no.~6, 641--651. \MR{2472470}

\bibitem{krueger.IA}
\bysame, \emph{Internal approachability and reflection}, J. Math. Log.
  \textbf{8} (2008), no.~1, 23--39. \MR{2674000}

\bibitem{mitchell}
W.~J. Mitchell, \emph{Aronszajn trees and the independence of the transfer
  property}, Ann. Math. Logic \textbf{5} (1972/73), 21--46. \MR{313057}

\bibitem{moore.MRP}
J.~T. Moore, \emph{Set mapping reflection}, J. Math. Log. \textbf{5} (2005),
  no.~1, 87--97. \MR{2151584}

\bibitem{moore.basis}
\bysame, \emph{A five element basis for the uncountable linear orders}, Ann. of
  Math. (2) \textbf{163} (2006), no.~2, 669--688. \MR{2199228}

\bibitem{Nee08}
I.~Neeman, \emph{Hierarchies of forcing axioms. {II}}, J. Symbolic Logic
  \textbf{73} (2008), no.~2, 522--542. \MR{2414463}

\bibitem{schimmerling}
E.~Schimmerling, \emph{Combinatorial principles in the core model for one
  {W}oodin cardinal}, Ann. Pure Appl. Logic \textbf{74} (1995), no.~2,
  153--201. \MR{1342358}

\bibitem{STR10}
R.~Strullu, \emph{{MRP}, tree properties and square principles}, forthcoming.

\bibitem{todorcevic.note_on_PFA}
S.~Todor{\v{c}}evi{\'c}, \emph{A note on the proper forcing axiom}, Axiomatic
  set theory (Boulder, Colo., 1983), Contemp. Math., vol.~31, Amer. Math. Soc.,
  Providence, RI, 1984, pp.~209--218. \MR{763902}

\bibitem{VIA10}
M.~Viale, \emph{On the notion of guessing model}, forthcoming.

\bibitem{weiss}
C.~Wei{\ss}, \emph{The combinatorial essence of supercompactness}, forthcoming.

\bibitem{diss}
\bysame, \href{http://edoc.ub.uni-muenchen.de/11438/}{Subtle and ineffable
  tree properties}, Ph.D. thesis, Lud\-wig {M}axi\-milians
  {U}ni\-versi\-t{\"a}t {M}{\"u}n\-chen (2010).

\bibitem{woodin}
W.~H. Woodin, \emph{The axiom of determinacy, forcing axioms, and the
  nonstationary ideal}, de Gruyter Series in Logic and its Applications,
  vol.~1, Walter de Gruyter \& Co., Berlin, 1999. \MR{1713438}

\end{thebibliography}

\end{document}